\documentstyle[amssymb,amsfonts]{amsart}

\def\P{{\hbox{\bf P}}}
\def\E{{\hbox{\bf E}}}

\def\F{{\hbox{\bf F}}}
 at 10 true pt

\def\be#1{ \begin{equation}\label{#1} }

\def\bas{\begin{align*}}
\def\eas{\end{align*}}
\def\bi{\begin{itemize}}
\def\ei{\end{itemize}}

\def\Z{{\hbox{\bf Z}}}

\newenvironment{proof}{\noindent {\bf Proof} }{\endprf\par}
\def \endprf{\hfill  {\vrule height6pt width6pt depth0pt}\medskip}
\def\emph#1{{\it #1}}
\def\textbf#1{{\bf #1}}

\def\ep{{\epsilon}}

\def \ep{\varepsilon}

\theoremstyle{plain}
  \newtheorem{theorem}[subsection]{Theorem}

  \newtheorem{proposition}[subsection]{Proposition}
  
  \newtheorem{lemma}[subsection]{Lemma}

  \newtheorem{claim}[subsection]{Claim}
\theoremstyle{remark}

\theoremstyle{definition}

\include{psfig}

\begin{document}

\title[]
{On distribution of three-term arithmetic progressions in sparse
subsets of $\F_p^n$}

\author{Hoi H. Nguyen}
\address{Department of Mathematics, Rutgers University, Piscataway, NJ 08854, USA}

\email{}
\thanks{}
\maketitle

\begin{abstract}
We prove a structural version of Szemer\'edi's regularity lemma  for subsets of a typical random set in $\F_p^n$. As an application, we give a short proof for an analog of a hard theorem by Kohayakawa, {\L}uczak, and R\"odl on the distribution of three-term arithmetic progressions in sparse sets.
\end{abstract}

\section{Introduction}\label{section:introduction}
\noindent Let $G$ be a graph and let $A,B$ be two subsets of
$V_G$. We define the density $d(A,B)$ of $G(A,B)$ to be

$$d(A,B):=e(A,B)/|A||B|.$$

\noindent Let $\ep$ be a positive constant. We say that the pair $(A,B)$ is {\it
$\ep$-regular} if

$$|d(A',B')-d(A,B)|\le \ep$$

\noindent for any $A'\subset A$ and $B'\subset B$ satisfying $|A'|\ge \ep |A|$ and $|B'|\ge \ep|B|$.

 Szemer\'edi's regularity lemma, a fundamental result in
combinatorics, states that the vertex set of any dense graph can be partitioned into
not-too-small pieces so that almost all pairs of pieces are regular.

\begin{theorem}[Szemer\'edi's regularity
lemma]\label{theorem:Szem:dense} Let $\ep>0$. There exists
$M=M(\ep)$ such that the vertex set can be partitioned into $1/\ep
\le m\le M$ sets $V_i$ with sizes differing by at most 1, such that
at least $(1-\ep)m^2$ of the pairs $(V_i,V_j)$ are $\ep$-regular.
\end{theorem}

 Consider a vector space $V=\F_p^n$, where $p$ is a
fixed odd prime and $n$ is a large integer. Let $A$ be a subset of $V$,
we define the (bipartite, directed) Cayley graph generated by $A$ to be
$G_A = G(V_1,V_2)$, where $V_1,V_2$ are two copies of $V$, and
$(v_1,v_2)\in E(G_A)$ if $v_2-v_1\in A$.

 It is clear that $G_A$ is a regular graph of degree $|A|$.
Hence if $A$ is dense enough, then Szemer\'edi's regularity lemma is
applicable to $G_A$. Furthermore, since $G_A$ has additional algebraic structure, it
is natural to expect a stronger result than Theorem \ref{theorem:Szem:dense}. Indeed, a
result of Green \cite[Section 9]{Gr} confirms this intuition:

\vskip2mm

{\it Assume that $|A|=\Omega(|V|)$. Then one can partition $V(G_A)$ into
affine subspaces of large dimension and so
that almost all pairs of subspaces are $\ep$-regular.}

\vskip2mm

 Szemer\'edi's regularity lemma is not meaningful for sparse
graphs in general. However, it can be extended to certain
graph families. Let $\ep$ be a positive constant. We say
that the pair $(A,B)$ is {\it relatively $\ep$-regular} if

$$|d(A',B')-d(A,B)|\le \ep d(G)$$

\noindent for any $A'\subset A$ and $B'\subset B$ satisfying $|A'|\ge \ep |A|$ and
$|B'|\ge \ep|B|$.

\noindent  Let be given $b>2$ and $\sigma>0$. We say that a graph
$G$ is {\it $(b,\sigma)$-sparse} if

$$d(X,Y)\le bd(G)$$

\noindent for any $|X|\ge \sigma |V_G|$ and $|Y|\ge \sigma|V_G|$.
The following result extends Szemerer\'edi' s regularity lemma for
$(b,\sigma)$-sparse graphs.

\begin{theorem}[Szemer\'edi's regularity lemma for sparse
graphs, {\cite[Lemma 4]{Ko-Lu-Ro}}]\label{theorem:Szem:sparse} Let $b>0$. For
$\ep>0$ there exists $\sigma=\sigma(b,\ep)$ such that the following
holds for all $(b,\sigma)$-sparse graphs. There exists $M=M(\ep,b)$
such that the vertex set can be partitioned into $1/\ep \le m\le M$
sets $V_i$ with sizes differing by at most 1, such that at least
$(1-\ep)m^2$ of the pairs $(V_i,V_j)$ are relatively $\ep$-regular.
\end{theorem}

 As to how Theorem \ref{theorem:Szem:sparse} extends Theorem
\ref{theorem:Szem:dense}, our first goal is to point out that the
result of Green can be extended easily to ``$(b,\sigma)$-sparse"
Cayley graphs in $\F_p^n$:

\vskip .2in

{\it Assume that $A$ is not too sparse, and $G_A$ is
$(b,\sigma)$-sparse with some reasonable constants $b,\sigma$. Then
Theorem \ref{theorem:Szem:sparse} is applicable to $G_A$ in such a
way that the vertex partitions can be taken to be affine subspaces of
high dimension.}

\vskip .2in

 We shall give a precise statement in Section
\ref{section:regularity}. Next, let $Z$ be an additive group and let $\alpha$ be a positive constant. We say that
a subset $R$ of $Z$ is $(\alpha,3AP)$-dense if any subset of $A$ of cardinality at
least $\alpha |A|$ must contain a nontrivial three-term
arithmetic progression in $Z$. It has been shown in \cite{Ko-Lu-Ro} that almost every subset of cardinality $\gg_{\alpha}|Z|^{1/2}$ of the cyclic group $Z=\Z_n$, where $n$ is odd, is $(\alpha, 3AP)$-dense.  Our next goal is to prove a similar result.

\begin{theorem}[Kohayakawa-{\L}uczak-R\"odl theorem for $\F_p^n$]\label{theorem:main}
There exists a constant $C=C(\alpha)$ such that the following holds for all $r\ge C(\alpha)|V|^{1/2}$. Let $R$ be a random subset of size $r$ of $\F_p^n$, then the probability that $R$ fails to be $(\alpha,3AP)$-dense is $o(1)$.
\end{theorem}

 To prove Theorem \ref{theorem:main} we follow the approach of \cite{Ko-Lu-Ro}. However, with our structure result in hand (Theorem \ref{theorem:regularity:1}), we are able to get around many technical difficulties to provide a much simpler proof.

\section{Notation}

\noindent {\bf Fourier transform.}(cf. \cite[Chapter 4.] {Tao-Vu}) Let $H$ be a subspace of $V$, let $f$ be a
real-valued function defined on $V$. Then the Fourier transform of
$f$ with respect to $H$ is

$$\widehat{f}(\xi):= \E_{x\in H} f(x)e(-\langle x,\xi \rangle)$$

\noindent Where $\langle x,\xi \rangle=\sum_{i=1}^n x_i\xi_i/p$, and $e(z)=e^{2\pi
i z}$.

\noindent {\bf Convolution.} Let $f$ and $g$ be two real-valued functions
defined on $V$. The convolution of $f$ and $g$ with respect to $H$
is

$$f*g(h):=\E_{x\in H}f(x)g(h-x).$$

\noindent The following basic properties for real-valued functions will be used several times.

\begin{itemize}

\item (Parseval's identity) $\E_{x\in H}f^2(x) = \sum_{\xi\in H}|\widehat{f}(\xi)|^2.$

\vskip .2in

\item (Plancherel's formula) $\E_{x\in H}f(x)g(x) = \sum_{\xi\in H}\widehat{f}(\xi)\overline{\widehat{g}(\xi)}.$

\vskip .2in

\item (Fourier inversion formula) $f(x)=\sum_{\xi\in H}\widehat{f}(\xi)e(\langle x,\xi \rangle)$.

\vskip .2in

\item $\widehat{f*g}(\xi)=\widehat{f}(\xi)\widehat{g}(\xi)$.

\end{itemize}

\noindent Let $A$ be a subset of $V$, and let $v$ be an element of
$V$. We define $A_H^v$ to be the set $A+v \cap H$. Sometimes we also
write $A_H^v$ as its characteristic function. Following are some simple properties:

\begin{itemize}
\item $\widehat{A_H^v}(\xi)=|A_H^v|/|H|$
if $\xi \in H^{\perp}$;

\vskip .15in

\item $\widehat{A_H^{v'}}(\xi)=e(\langle v-v',\xi \rangle)\widehat{A_H^v}(\xi)$ if $v-v'\in H$; in particular, $|\widehat{A_H^{v'}}(\xi)|=|\widehat{A_H^v}(\xi)|$.
\end{itemize}

\noindent {\bf $\ep$-regular vector.} Let $\ep$ be a positive constant. Let $A$ be a given set. We say
that a vector $v$ is an $\ep$-regular vector with respect to $H$ if

$$\sup_{\xi \notin H^{\perp}} |\widehat{A_H^v}(\xi)| \le \ep |A|/|V|.$$

\noindent (It is more natural to use the upper bound $\ep |A_H^v|/|H|$ in the definition above, but we find our definition more convenient to use, and $\ep |A|/|V|$ is the typical value for  $\ep |A_H^v|/|H|$.)

\noindent Notice that if $v$ is an $\ep$-regular vector, then so is
any element of $v+H$.

\noindent We say that a subspace $H$ is {\bf $\ep$-regular for} $A$ if
the number of $v$'s which fail to be $\ep$-regular is at most $\ep
|V|$.

\noindent {\bf $\sigma$-regular set.} Let $\sigma$ be a positive constant. We say that a subset $R$ of $V$ is $\sigma$-regular if the number of edges between $X$ and $Y$ in the Cayley graph $G_R$ is as many as expected,

$$e_{G_R}(X,Y)=(1+o_{\sigma}(1))|R||X||Y|/|V|,$$

provided that $|X|,|Y|\ge \sigma N$.

\noindent Roughly speaking, a typical large random set is
$\sigma$-regular for quite small $\sigma$ (see Subsection \ref{subsection:lemmas:random}). In particular, the set
$V$ itself is $\sigma$-regular for all $\sigma$.

\noindent {\bf Dependency of constants.} We shall work with several constants throughout this note, so let us mention briefly here to avoid confusion.

$$\alpha, c(\alpha)\rightarrow \eta \rightarrow \ep \rightarrow \sigma \rightarrow C $$

\noindent First, $\alpha$ is the constant that we fix all the time. The constants $c(\alpha)$'s depend only on $\alpha$; these constants will appear as exponents in Section \ref{section:application}. Next, $\eta$ will be chosen to be small enough depending on $\alpha$ and the $c(\alpha)$'s. The constant $\ep$ will be considered as an arbitrary constant in Section \ref{section:regularity} and Subsection \ref{subsection:lemmas:random}, but it will depend on $\alpha$ and $\eta$ in later sections. Last but not least, $\sigma$ depends on $\alpha$ and $\ep$. We shall choose $\eta,\ep,\sigma$ to be small enough, while constants $C=C(\alpha, \eta, \ep,\sigma)$ are often very large.

\noindent{\bf Tower-type function.} We shall use a tower-type function $W(t)$ defined
recursively by $W(1)=2p$ and $W(t)=(2p)^{W(t-1)}$ for $t\ge 2$.

\noindent The note is organized as follows. In Section
\ref{section:regularity} we discuss about Green's result for sparse
Cayley graphs. Next in Section \ref{section:lemmas} we shall provide some
ingredients for applications. The proof of Theorem \ref{theorem:main} is established in Section
\ref{section:application}.

\section{Green's regularity lemma for $(b,\sigma)$-sparse
Cayley graphs}\label{section:regularity}

 In this section, unless otherwise specified, all Fourier transforms
and convolutions are taken with respect to an underlying subspace
$H$. For short, we let $N=|V|$.

\begin{theorem}\label{theorem:regularity:1}
Let $\alpha, \ep \in (0,1)$. There is a constant
$\sigma=\sigma(\ep,\alpha)$ such that the following holds. Let $R$
be a $\sigma$-regular set of $V$ and let $A$ be a subset of $R$ of cardinality
$\alpha |R|$. Then there is a subspace $H \le V$ of index
at most $W(4 (\ep \alpha)^{-2})$ which is $\ep$-regular for $A$.
\end{theorem}

 We pause to discuss the strength of Theorem \ref{theorem:regularity:1}. First, since $R$ is $\sigma$-regular, the Cayley graph $G_A$ generated by $A$ is $(2/\alpha,\sigma)$-sparse. Indeed, for any $X,Y \in
V$ such that $|X|\ge \sigma |V|$ and $|Y|\ge |V|$ we have $
e_{G_A}(X,Y)\le e_{G_R}(X,Y)$. On the other hand, since $R$ is
$\sigma$-regular, we have

\begin{align*}
e_{G_R}(X,Y) &= (1+o_{\sigma}(1))|R||X||Y|/N\\
&\leq 2|R||X||Y|/N \\
&\leq (2/\alpha)|A||X||Y|/N \\
&= (2/\alpha) |X||Y|d(G_A).
\end{align*}

\noindent Hence, $d_{G_A}(X,Y)\le (2/\alpha) d(G_A).$

 Now, since $G_A$ is $(2/\alpha,\sigma)$-sparse, Theorem
\ref{theorem:Szem:sparse} is applicable to $G_A$. The advantage of
Theorem \ref{theorem:regularity:1} is, besides implying Theorem
\ref{theorem:Szem:sparse}, it also provides a well-structured
partition for the vertex set of $G_A$ as follows.

 Let $V=\cup_{i=1}^K H_i$ be the partition of $V$ into
affine translates of $H$. Let $v_1,\dots,v_K$ be representatives of
the coset subgroups $V/H$. Then by definition, all but at most $\ep
K$ vectors $v_1,\dots,v_K$ are $\ep$-regular vectors with respect to $H$.

 Next assume that $H_i=v_i+H$ and $H_j=v_j+H$ are two affine
translates of $H$ such that $v_j-v_i$ is an $\ep$-regular vector. We will show that the subgraph
$G_A(H_i,H_j)$ is relatively $\ep^{1/3}$-regular.

 It is clear that $e_{G_A}(H_i,H_j)= |H||A_H^{v_j-v_i}|$; thus

$$d_{G_A}(H_i,H_j)=|A_H^{v_j-v_i}|/|H|.$$

\noindent Let $X\subset H_i$ and $Y\subset H_j$ be any two subsets of
$H_i$ and $H_j$ respectively, which satisfy $|X|,|Y|\ge \ep^{1/3} |H|$.
We shall estimate the number of edges generated by $X$ and $Y$. We have

\begin{align*}
e_{G_A}(X,Y) &= \sum_{x\in H_i,y\in H_j} A(y-x)X(x)Y(y)\\
&=\sum_{x',y'\in H} A_H^{v_j-v_i}(y'-x')X(x'+v_i)Y(y'+v_j) \\
&= \sum_{x',y'\in H} A_H^{v_j-v_i}(y'-x')\left (X-v_i \right )(x')(Y-v_j)(y')
\end{align*}

\noindent Now we apply the Fourier inversion formula to the last sum,

\begin{align*}
e_{G_A}(X,Y) &=|H|^2\sum_{\xi\in H}\widehat{A_H^{v_j-v_i}}(\xi)\widehat{(X-v_i)}(-\xi)\widehat{(Y-v_j)}(\xi)\\
&=|A_H^{v_j-v_i}||X||Y|/H+\sum_{\xi\in H\backslash \{0\}}\widehat{A_H^{v_j-v_i}}(\xi)\widehat{(X-v_i)}(-\xi)\widehat{(Y-v_j)}(\xi).
\end{align*}

\noindent Since $v_j-v_i$ is an $\ep$-regular vector with respect to $H$, we infer that

$$\left |e_{G_A}(X,Y)-|A_H^{v_j-v_i}||X||Y|/|H| \right | \leq (\ep |A^{v_j-v_i}|/N)\sum_{\xi}\left |\widehat{(X-v_i)}(-\xi)\widehat{(Y-v_j)}(\xi)\right |.$$

\noindent By Parseval's identity and by the Cauchy-Schwarz inequality we thus have

\begin{align*}
\left |e_{G_A}(X,Y)-|A_H^{v_j-v_i}||X||Y|/|H| \right | &\le |H|(\ep|A^{v_j-v_i}|/N) (|X||Y|)^{1/2}\\
&\leq \ep |A^{v_j-v_i}||H|^2/N.
\end{align*}

\noindent It follows that

\begin{align*}
|d_{G_A}(X,Y)-d_{G_A}(H_i,H_j)| &\le \ep |A^{v_j-v_i}||H|^2/ (|X||Y|N)\\
&\le \ep^{1/3} |A^{v_j-v_i}|/N \\
&= \ep^{1/3} d(G_A).
\end{align*}

 Hence $G_A(H_i,H_j)$ is indeed relatively $\ep^{1/3}$-regular. One
observes that $v_j-v_i$ is an $\ep$-regular vector for all but at most $\ep K^2$ pairs $(i,j)$. Hence the partition $V=\cup_{i=1}^K
H_i$ satisfies the conclusion of Theorem \ref{theorem:Szem:sparse}.

 Another crucial observation, which will be used later on in applications, is that the definition of $\ep$-regular vector works for any type of (linear) Cayley graph. For instance assume that $(v_1+v_2)/2$ is an $\ep$-regular vector with respect to $H$ and define a bipartite Cayley graph $G_A'$ on $(H-v_1,H-v_2)$ by connecting $(h_1-v_1)$ with $(h_2-v_2)$ if $((h_1-v_1)+(h_2-v_2))/2=(h_1+h_2)-(v_1+v_2)/2 \in A$. Then this graph $G_A'$ is also $\ep^{1/3}$-regular. To justify this fact, the reader just needs to follow the same lines of verification used for $G_A$ above.

 Now we start to prove Theorem \ref{theorem:regularity:1}.

\begin{proof}. Define $d(A,H)$ by

$$d(A,H):=\frac{1}{N}\sum_{v \in V}\left (\frac{|A_H^v|}{H}\right )^2/\left (\frac{|A|}{N}\right )^2.$$

\noindent Observe that $d(A,H)$ is the mean of the squares of the normalized densities of the $G_A(H_i,H_j)$'s. We show that this quantity is always bounded.

\begin{claim}\label{claim:regularity:2}
We have $d(A,H)\le 4/\alpha^2$ for any $|H|\ge \sigma N$.
\end{claim}

\begin{proof}(of Claim \ref{claim:regularity:2}). Since $H\ge \sigma N$, by the
$\sigma$-regularity of $R$, for any $v$ we have,

$$|H||R_H^v|=e_{G_R}(H,H-v)=(1+o_{\sigma}(1))|H||H||R|/N.$$

\noindent Hence $|A_H^v|/|H| \le |R_H^v|/|H|\le 2|R|/N \le (2/\alpha)|A|/N$. As a result,

$$d(A,H) \le \frac{1}{N} \sum_{v\in V} (2/\alpha)^2 \le 4/\alpha^2.$$

\end{proof}

 As in the proof of Szemer\'edi's regularity lemma, when a
partition with too many irregular pairs comes into play, then we
pass to a finer partition, and by so the mean square of
the densities will increase. What we are going to do is
similar, the only difference is we restrict ourselves to a special
family of partitions.

\begin{lemma}\label{lemma:regularity:3}
Let $\ep \in (0,1)$ and suppose that $H$ is a subspace of $V$, which
is not $\ep$-regular for $A$. Then there is a subspace $H'\le H$
such that $|V/H'|\le (2p)^{|G/H|}$ and $d(A,H')\ge d(A,H)+\ep^3$.
\end{lemma}

\begin{proof}(of Lemma \ref{lemma:regularity:3}).
Since $H$ is not $\ep$-regular for $A$, there are $\ep N$ vectors $v$ such that $\sup_{\xi
\notin H^{\perp}}|\widehat{A_H^v}(\xi)| \ge \ep |A||H|/N$. In other
words, there exists a positive integer $m$ satisfying $\ep N/|H| \le m \le N/|H|$ together with $m$ coset representatives $v_1,\dots, v_m\in V/H$ and vectors $\xi_1,\dots, \xi_m\in H$, where $N/|H| \ge m \ge \ep|N|/|H|$, such that

$$|\widehat{A_H^{v_i}}(\xi_i)|\ge \ep |A|/N.$$

Now let $H'\subset H$ be the annihilator of all $\xi_i$'s. It is
clear that

$$|H'|\ge |H|/p^m \ge |H|/p^{|V/H|}$$

\noindent Hence,

$$|V/H'|\le |V/H|p^{|V/H|} < (2p)^{|V/H|}.$$

\noindent Set $S: = N|H'|^2 (|A|/N)^2 |H| d(A,H')$. It is obvious that

$$S = |H|\sum_{v\in V}|A_{H'}^v|^2=\sum_{v\in V,h\in H}|A_{H'}^{v+h}|^2.$$

\noindent Notice that $|A_{H'}^{v+h}| = \sum_{x\in H}(A+v)(x-h)H'(x) = \sum_{x\in H}(A+v)(x)H'(x+h) = |H|(A_H^{v}*H')(-h)$. We rewrite $S$ and then use Plancherel's formula,

\begin{align*}
S &=|H|^2\sum_{v\in V,h\in H}|A_H^v*H'(h)|^2\\
&= |H|^3 \sum_{v\in V, \xi\in H}\left | \widehat{A_H^v*H'}(\xi)\right |^2\\
&=|H|^3 \sum_{v\in V, \xi\in H} |\widehat{A_H^v}(\xi)|^2 |\widehat{H'}(\xi)|^2.
\end{align*}

\noindent In the last sum, the contribution of the $\xi=0$ term gives

\begin{align*}
S_0 &= |H|^3 \sum_{v\in V}\left (|A_H^v|/|H|\right)^2 \left (|H'|/|H|\right )^2\\
&= |H||H'|^2 \sum_{v\in V}\left (|A_H^v|/|H|\right )^2 \\
&= N|H||H'|^2(|A|/N)^2d(A,H);
\end{align*}

\noindent while the sums contributed from $\xi\in H\backslash \{0\}$
is bounded from below by

$$S_{\neq 0} \ge |H|^3\sum_{i=1}^m\sum_{v\in H+v_i}|\widehat{A_H^{v_i}}(\xi_i)|^2
|\widehat{H'}(\xi_i)|^2.$$

\noindent But since $\xi_i\in {H'}^{\perp}$, we have
$\widehat{H'}(\xi_i)=|H'|/|H|$. Use the bound $|\widehat{A}_H^{v_i}(\xi_i)| \ge \ep |A|/N$ for all $1\le i\le m$, we obtain

\begin{align*}
S_{\neq 0} &\ge |H|^3 m|H| (\ep |A|/N)^2 \left(|H'|/|H|\right )^2\\
&\ge |H|^3 (\ep |N|/|H|) |H| \ep^2 \left (|A|/N \right )^2 \left (|H'|/|H|\right )^2\\
&= \ep^3 |H||H'|^2 N \left(|A|/N\right )^2.
\end{align*}

\noindent From the estimate for $S_0$ and $S_{\neq 0}$ we conclude
that $d(A,H')\ge d(A,H)+\ep^3.$

\end{proof}

 To complete the proof of Theorem
\ref{theorem:regularity:1} we keep applying Lemma
\ref{lemma:regularity:3}. Since $d(A,H) \le 4/\alpha^2$, the
iteration stops after at most $4 \ep^{-3}\alpha^{-2}$ steps. During the
iteration, $|H'|$ is always bounded below by $N/W(4
\ep^{-3}\alpha^{-2})$, thus we may choose $\sigma= (2W(4
\ep^{-3}\alpha^{-2}))^{-1}$.

\end{proof}

\noindent Let us conclude this section by mentioning an important corollary of Theorem \ref{theorem:regularity:1}.

\begin{theorem}\label{theorem:regularity:2}
Let $\alpha, \ep \in (0,1)$ and let $m$ be a positive integer. There is a constant
$\sigma=\sigma(\ep,\alpha,m)$ such that if $R$
is a $\sigma$-regular set of $V$ and $A$ is a subset of $R$ of cardinality
$\alpha |R|$, then the following holds. Assume that $A=\cup_{i=1}^m A_i$ is a partition of $A$ into $m$ distinct sets of size $|A|/m$.
Then there is a subspace $H \le V$ of index bounded by $W(4m^2\ep^{-3}\alpha^{-2})$ which is $\ep$-regular for all $A_i$'s.
\end{theorem}

 To prove Theorem \ref{theorem:regularity:2} first we let $d(A_1,\dots,A_m,H):=\sum_{i=1}^m d(A_i,H)$. Next keep iterating Lemma \ref{lemma:regularity:3} if $H$ is not $\ep$-regular for some $A_i$. Since  $d(A_1,\dots,A_i,H) \le 4m^2/\alpha^2$, the iteration will stop after at most $4m^2\ep^{-3}\alpha^{-2}$ steps.

\section{Main lemmas for applications}\label{section:lemmas}

\subsection{Regularity of a random set}\label{subsection:lemmas:random}

\begin{lemma}\label{lemma:lemmas:random:1}
 For $\sigma>0$ there is a constant $C(\sigma)$ such that if $r\ge C(\sigma) N^{1/2}$ and $R$ is a random subset of size $r$ of $V$, then $R$ is a $\sigma$-regular set almost surely.
\end{lemma}

\noindent To start with, we consider a slightly different model as
follows.

\begin{lemma}\label{lemma:lemmas:random:2}
For $\sigma>0$ there is a constant $C(\sigma)$ such that if $r\ge C(\sigma) N^{1/2}$ and $q=r/N$, and $R$ is
a subset of $V$ whose elements are equally selected with probability $q$, then
$R$ is a $\sigma$-regular set almost surely.
\end{lemma}

\begin{proof} (of Lemma \ref{lemma:lemmas:random:2})
 Let $X,Y \subset V$, of cardinality at least $\sigma N$.
The number of edges of $G_R$ generated by $X$ and $Y$ is

$$e_R(X,Y)=\sum_{x,y\in V}1_{R}(y-x)1_X(x)1_Y(y)=N^2\sum_{\xi\in
V}\widehat{1_{R}}(\xi)\widehat{1_X}(\xi)\widehat{1_Y}(-\xi)$$

\noindent where the Fourier transform is defined with respect to $V$, and the
latter identity comes from Fourier inversion formula. Thus we have

$$e_R(X,Y)=|R||X||Y|/N + N^2\sum_{\xi\in V,\xi \neq 0}\widehat{1_{R}}(\xi)\widehat{1_X}(\xi)\widehat{1_Y}(-\xi).$$

\noindent Let us pause to estimate $\widehat{1_R}(\xi)$.

\begin{lemma}\label{lemma:lemmas:random:3}
$\sup_{\xi\neq 0}|\widehat{1_{R}}(\xi)|<|R|/(N\log N)$ almost surely
for $R$.
\end{lemma}

 The proof of this lemma is routine by applying the exponential moment method. For the sake of completeness, we prove it in Appendix A.

 Assuming Lemma \ref{lemma:lemmas:random:3}, then by the Cauchy-Schwarz inequality and Parseval's identity we have

\begin{align*}
\left |e_R(X,Y)-|R||X||Y|/N \right | &\le N^2 \sup_{\xi\neq 0}|\widehat{1_R}(\xi)| \left (\sum_{\xi\in V}|\widehat{1_X}(\xi)|^2\sum_{\xi\in V}|\widehat{1_Y}(\xi)|^2 \right )^{1/2}\\
&\le N^2 \sup_{\xi\neq 0}|\widehat{1_R}(\xi)| \left (|X||Y|/N^2 \right )^{1/2}\\
&=\sup_{\xi\neq 0}|\widehat{1_R}(\xi)| (|X||Y|)^{1/2}N.
\end{align*}

 On the other hand, as $|X|,|Y|\ge \sigma N/4$ and
$\sup_{\xi\neq 0}|\widehat{1_R}(\xi)| \le |R|/(N \log N)$, we have

$$\sup_{\xi\neq 0}|\widehat{1_R}(\xi)| (|X||Y|)^{1/2}=o\left (|R||X||Y|/N \right),$$

\noindent completing the proof of Lemma \ref{lemma:lemmas:random:2}.

\end{proof}

 Next we show that the two models, of Lemma
\ref{lemma:lemmas:random:1} and of Lemma \ref{lemma:lemmas:random:2}, are similar.

\begin{proof}(of Lemma \ref{lemma:lemmas:random:1}).
 Let $q=(1-\sigma^4)|R|/N$. We first consider a random set
$R_1$ by selecting each element of $V$ with probability $q$. It is
obvious that the size of this random set belongs to
$[(1-2\sigma^4)|R|,|R|]$ almost surely. We restrict ourself to this
event by renormalizing the probability space. Hence the random set
$R_1$ is chosen uniformly from the collection of subsets of size
$[(1-2\sigma^4)|R|,|R|]$. Next we pick uniformly a set $R_2$ of size
$|R|-|R_1|$ from $V\backslash R_1$ and set $R=R_1\cup R_2$.

 Suppose that $X,Y\subset V$ and $|X|,|Y|\ge \sigma N$.
Since $R_1$ is $\sigma$-regular almost surely by Lemma \ref{lemma:lemmas:random:2}, we have
$(1-o_{\sigma}(1))|R_1||X||Y|/N \le e_{R_1}(X,Y) \le
(1+o_{\sigma}(1))|R_1||X||Y|/N$. On the other hand, it is obvious
that

\begin{align*}
e_{R_1}(X,Y) &\le e_{R}(X,Y)\\
&\le e_{R_1}(X,Y) + |R_2|N\\
&\le (1+o_\sigma(1))|R||X||Y|/N + 2\sigma^4 |R|N\\
&= (1+o_\sigma(1))|R||X||Y|/N.
\end{align*}

\noindent Hence $e_R(X,Y)=(1+o_\sigma(1))|X||Y|/N$ almost surely,
completing the proof of Lemma \ref{lemma:lemmas:random:1}.
\end{proof}

\subsection{Edge distribution of quasi-random graphs}\label{subsection:lemmas:counting} Roughly speaking, if we choose randomly a large number of vertices of a dense quasi-random graph, then the chance of obtaining an edge is very high. This simple observation, as a strong tool to exploit structure for counting, was used in \cite{Ko-Lu-Ro}, and will play a key role in our proof of Theorem \ref{theorem:main}.

 Let $G = G(u,\rho,\ep)$ be an $\ep$-regular bipartite graph, $V(G) = U_1 \cup U_2$,
where $|U_1| = |U_2| = u$ and $d(G) = e(G)/u^2 = \rho$. Let $t_1, t_2<u/2$ be two given positive integers.
We select a random subgraph of $G$ as follows. First, an adversary chooses a set $S_1 \subset U_1$ with $|S_1|\le u/2$. Then we
pick a set $T_1\subset U_1 \backslash S_1$ with $|T_1| = t_1$ from the collections of all $d_1$-subsets of $U_1 \backslash S_1$ with equal probability.
Next, our adversary picks a set $S_2  \subset U_2$
with $|S_2|\le u/2$, and we pick a set $T_2\subset U_2\backslash S_2$ with $T_2= t_2$ from the collections of all $t_2$-subsets of $U_2 \backslash S_2$ with equal probability. Let us call the outcome of the above procedure a random $(t_1,t_2)$-subgraph of $H$.

\begin{lemma}\label{lemma:KLR:11}\cite[Lemma 11]{Ko-Lu-Ro}\label{lemma:lemmas:counting:4}
For every constant $0 < \eta< 1$, there exist a constant $0 < \ep< 1$ and a natural number $u_0$ such that, for any real
$t\ge 2(u/\ep)^{1/2}$ and any given graph $G = G(u,\rho,\ep)$ as above with $u \ge u_0$
and $\rho\ge t/u$, the following holds. If $t_1, t_2 \ge t$,
regardless of the choices for $S_1$ and $S_2$ of our adversary, the
probability that a random $(t_1,t_2)$-subgraph of $G$ fails to
contain an edge is at most $\eta^t$.
\end{lemma}

\noindent The proof of Lemma \ref{lemma:lemmas:counting:4} is simple, the interested reader may read \cite{Ko-Lu-Ro}.

\subsection{Roth's theorem for $\F_p^n$} Another important ingredient which will serve as a starting point for our argument is Roth's theorem.

\begin{theorem}\label{theorem:application:2} For any $\delta>0$ there is a number $c(\delta)>0$ such that
if $B$ is a subset of $V$ of size $\delta |V|$, then $B$ contains at least $c(\delta)|V|^2$ three-term arithmetic progressions.
\end{theorem}

In the next section, we shall put every thing together to establish Theorem \ref{theorem:main}.

\section{Proof of Theorem \ref{theorem:main}}\label{section:application}

First, by Theorem \ref{theorem:application:2}, it is enough to work with the case

$$r=o_{\alpha}(N).$$

\noindent We say that a set $A$ which belongs to some $\sigma$-regular set $R$ is {\it $(\alpha,\sigma)$-bad} if $|A|=\alpha |R| = \alpha r$ and it contains no nontrivial three-term arithmetic progression.
Our main goal is to give an upper bound for the number of bad sets of a given size.

\begin{theorem}\label{theorem:application:1} Let $\alpha$ and $\eta$ be given positive numbers. Then there exist constants $c=c(\alpha)>0, C=C(\eta,\alpha)>0$
and $\sigma=\sigma(\alpha,\eta)>0$ such that for all $s\ge C(\alpha,\eta)N^{1/2}$, the number of $(\alpha,\sigma)$-bad sets of
size $s$ is at most $\eta^{c(\alpha)s} \binom{n}{s}$.
\end{theorem}

\begin{proof} (of Theorem \ref{theorem:main} assuming
Theorem \ref{theorem:application:1}). We choose $\eta=\eta(\alpha)$ to be small
enough. Let $s\ge  C(\alpha,\eta)N^{1/2}$ and put $r=s/\alpha$. Pick a random set $R$
among all $r$-subsets of $[n]$. Then by Theorem \ref{lemma:lemmas:random:1}, $R$ is $\sigma$-regular almost surely. Among these $\sigma$-regular $r$-sets, by Theorem \ref{theorem:application:1}, the number of sets that contain at least an $(\alpha,\sigma)$-bad subset is at most
$$\eta^{c(\alpha)s} \binom{n}{s}\binom{n-s}{r-s}.$$

 Observe that, as $\eta$ is small enough, this amount is $o \left (\binom{n}{r}\right )$. Hence almost all $r$-sets of $[n]$ contain no bad subsets at all. To finish the proof, we note that if $R$ contains no $(\alpha,\sigma)$-bad subset, then it is $(\alpha,3AP)$-dense.

\end{proof}

 We shall concentrate on proving Theorem \ref{theorem:application:1}  by localizing some properties of $A$. Our approach follows that of \cite{Ko-Lu-Ro} closely, but the key difference here is that we shall exploit rich structure obtained from Theorem \ref{theorem:regularity:1} and Theorem \ref{theorem:regularity:2}.

 Let $R$ be a $\sigma$-regular of fixed size $C(\sigma)N^{1/2} \le r = o(N)$ such that $A\subset R$. Let $m=m(\alpha)$ be a large number to be defined later.

 From now on we shall view $A$ as an ordered $m$-set-tuple, $A=(A_1,\dots, A_m)$ where $|A_i|=|A|/m$ for all $i$ and $A=\cup A_i$. We shall choose $\ep=\ep(\alpha)$ to be small enough.  By Theorem \ref{theorem:regularity:2}, there exists a subspace $H$ of $V$ which has index bounded by $W(4m^2\alpha^{-2}\ep^{-3})$ and which is $\ep$-regular for all $A_i$'s .

 Let $v_1,\dots, v_K$ be representatives of the coset subgroup $V':=V/H$. For each $A_i$, let us consider a set $B_i$ of vectors $v$ that satisfy the following conditions:

\begin{itemize}

\item $v$ is $\ep$-regular with respect to $A_i$ and $H$.

\vskip .2in

\item $|(A_i)_H^v| \ge (1/4) |A_i||H|/N$.

\end{itemize}

 It is clear that $|(A_i)_H^v|\le A_H^v \le |R_H^v|$. But by
definition of $R$, $|R_H^v| \le 2|R||H|/N$; thus we have

\begin{align*}
\sum_{v\in B_i} |(A_i)_H^v| &\ge |A_i|-(\ep K)(2|R||H|/N) - K((1/4) |A_i||H|/N)\\
&\ge(1-(\ep m)/\alpha -1/4)|A_i|\ge |A_i|/2,
\end{align*}

\noindent provided that $\ep\le \alpha/2m$. We infer that the size of $B_i$ is large,

$$|B_i|\ge (|A_i|/2)/(2|R||H|/N) \ge \frac{\alpha}{4m}K.$$

\noindent By a truncation if needed, we assume that $B_i$ has cardinality $(\alpha/4m) K$ for all $i$. Notice that these sets are not necessarily disjoint. We shall show that there are many three-term arithmetic progressions (in $V'$) with the property that all 3 terms belong to different $B_i$'s.

\noindent Now we set $B:=\{v\in V': v\in B_i\cap B_j\cap B_k \mbox{ for some } i<j<k\}$ and consider the following two cases.

{\bf Case 1}. $|B|\ge (\alpha/ 8m)K = (\alpha/8m) |V'|$. Applying Theorem \ref{theorem:application:2} we obtain $c(\alpha/8m)K^2$ three-term arithmetic progressions in $B$. By the definition of $B$, it follows that there are $c(\alpha/8m)K^2$ three-term arithmetic progressions with the property that all three terms belong to three different sets $B_i$.

{\bf Case 2}. $|B|\le (\alpha/ 8m)K = |B_i|/2$. We let $B' = \cup_{i=1}^m B_i \backslash B$. By an elementary counting argument, it follows that $|B'|\ge m|B_i|/4 = (\alpha/16)K$. Let us write $B'=\cup_{i=1}^m B_i'$, where $B_i'\subset B_i$ and all $B_i'$ are disjoint.

 By Theorem \ref{theorem:application:2}, the set $B'$ contains $c(\alpha/16) K^2$ three-term arithmetic progressions. Among them, since each three-term arithmetic progression is defined by two parameters,  the number of three-term arithmetic progressions that consist of at least two terms from the same $B_i'$ is bounded by $3\sum_{i=1}^m |B_i'|^2$. The latter quantity is bounded by $3|B_i| (\sum_{1}^m |B_i|) \le 3(\alpha/4m)(\alpha/4)K^2$; which is negligible compared to $c(\alpha/16) K^2$ by letting $m=m(\alpha)$ large.

 In both cases, the number of three-term arithmetic progressions with the property that all three terms belong to three different sets $B_i$ is at least $c'(\alpha)K^2$, where $c'(\alpha) = \min(c(\alpha/8m), c(\alpha/16)/2)$. By an averaging argument, there exist  three indices $i_0<j_0<k_0$ such that the number of three-term arithmetic progressions in $B_{i_0}\times B_{j_0}\times B_{k_0}$ is at least $c'(\alpha) K^2 /m^2 = c''(\alpha) K^2$. In particular, there exist a vector $v_{i_0}\in B_{i_0}$ and $c''(\alpha)K$ pairs $(v_{j_0}^l,v_{k_0}^l)\in B_{j_0}\times B_{k_0}$ such that each triple $(v_{i_0},v_{j_0}^l,v_{k_0}^l)$ is a three-term arithmetic progression.

\noindent Let us summarize what have been achieved.

\begin{enumerate}

\item There exists a subspace $H$ of index bounded by a function of $\alpha$ and $\ep$, and there exist $A_{i_0},A_{j_0},A_{k_0}$  and triples $(v_{i_0},v_{j_0}^l,v_{k_0}^l)$, where $1\le l \le c'''(\alpha) K$, such that the following holds:

\vskip .2in

\item $v_{i_0}$ is an $\ep$-regular vector for $A_{i_0}$, and $(A_{i_0})_H^{v_{i_0}} \ge (1/4) |A_{i_0}||H|/N$;

\vskip .2in

\item $(A_{j_0})_H^{v_{j_0}^l}\ge (1/4)|A_{j_0}||H|/N = (1/4m)s|H|/N$;

\vskip .2in

\item $(A_{k_0})_H^{v_{k_0}}\ge (1/4)|A_{k_0}||H|/|V| = (1/4m)s|H|/N$;

\vskip .2in

\item $(v_{i_0},v_{j_0}^l,v_{k_0}^l)$ is a three-term arithmetic progression in $V/H$.

\end{enumerate}

 One also observes that $v_{j_0}^l,v_{k_0}^l$ are $\ep$-regular vectors with respect to $A_{j_0}$ and $A_{k_0}$; but we do not need this fact. Since this configuration arises from \cite{Ko-Lu-Ro}, let us call it an {\it $\left(\alpha,\ep, H, i_0,j_0,k_0,v_{i_0},v_{j_0}^l,v_{k_0}^l\right )$-flower}.  Roughly speaking, the reader may visualize a flower with a center $A_{i_0}+v_{i_0}\cap H$, where $A_{i_0}+v_{i_0}\cap H$ sits nicely in $H$, and with $c'''(\alpha)K$ petals $(A_{j_0}+v_{j_0}^l\cap H,A_{k_0}+v_{k_0}^l\cap H)$.

 We denote by ${\mathcal{S}}$ the collections of all ordered $m$-set-tuples $A=(A_1,\dots,A_m)$ of size $|A|=s$.

\begin{proposition}\label{proposition:application:3} Let $\alpha,\eta$ be given. Then there exist constants  $c=c(\alpha)>0, C=C(\alpha,\eta)$ and $\ep=\ep(\alpha, \eta)>0$ such that the number of ordered $m$-set-tuples $A=(A_1,\dots,A_m)$ of size $s$, where $s\ge C(\alpha,\eta)N^{1/2}$, that contain a flower but not any non-trivial three-term arithmetic progression is at most $\eta^{c(\alpha)s}|\mathcal{S}|$.
\end{proposition}

 It is clear that Proposition \ref{proposition:application:3} implies Theorem \ref{theorem:application:1}. Hence we just need to prove Proposition \ref{proposition:application:3}.

 First of all we shall estimate the probability that a set $A$ that contains a given $\left(\alpha,\ep, i_0,j_0,k_0,v_{i_0},v_{j_0}^l,v_{k_0}^l\right)$-flower but contains no non-trivial three-term arithmetic progressions. On this probability space we also fix $A_{i_0}$ and fix the size of $A_{j_0}^l\cap H$ and $A_{k_0}^l\cap H$ for all $l$. Hence, $v_{i_0}$ is an $\ep$-regular vector with respect to a fixed sets $A_{i_0}$ and $H$; and the sets $A_{j_0}$ and $A_{k_0}$ vary in such a way that $v_{j_0}^l, v_{k_0}^l$ satisfy (3) and (4) respectively (in other words, $A_{j_0},A_{k_0}$ intersects $H-v_{j_0}^l,H-v_{k_0}^l$ in sets of given size).

 Without loss of generality, we assume that $2v_{i_0} =v_{j_0}^l+v_{k_0}^l$ for all $l$. We define a Cayley graph between $H-v_{j_0}^l$ and $H-v_{k_0}^l$ by connecting $v_1\in H-v_{j_0}^l$ to $v_2 \in H-v_{k_0}^l$ if $(v_1+v_2)/2 \in A_{i_0}$. Since $v_{i_0}$ is $\ep$-regular with respect to $A_{i_0}$, by the observation made before proving Theorem \ref{theorem:regularity:1}, this graph is also $\ep$-regular.

  By choosing $\ep=\ep(\alpha,\eta)=\ep(\alpha)$ to be small enough, and recalling that $N^{1/2}\ll s<r = o(N)$, we may check that for each bipartite graph $(H-v_{j_0}^l,H-v_{j_0}^l)$, the assumptions of Lemma \ref{lemma:lemmas:counting:4} are satisfied with $S_1=\bigcup_{1 \le m<j_0}  (A_m\bigcap (H-v_{j_0}^l)), T_1 = A_{j_0}^l\bigcap (H-v_{j_0}^l), S_2=\bigcup_{1\le m<k_0}(A_m\bigcap (H-v_{k_0}^l))$, and $T_2 = A_{k_0}\bigcap (H-v_{k_0}^l)$. It follows that the probability each petal fails to contain a there-term arithmetic progression is less than $\eta^{(1/4)s/(mK)}$. Hence the probability that $A$ contains no non-trivial three-term arithmetic progressions is less than $\eta^{(1/4) c''(\alpha)s /m} =\eta^{c'''(\alpha)s}$.

 Now we bound the number of flowers: the number of choices for $H$ is bounded by $N^{W(4m^2\alpha^{-1/2}\ep^{-3})}$, the number of choices for $(i_0,j_0,k_0,v_{i_0},v_{j_0}^l,v_{k_0}^l)$ is bounded by $K^{4+2c''(\alpha) K}$ (which is independent of $N$). Hence there are at most $N^{C(\alpha)}$ flowers.

 Putting everything together, we infer that the number of $A$ that contains some flower but not any non-trivial three-term arithmetic progression is at most

$$N^{C(\alpha)} \eta^{c'''(\alpha)s} |{\mathcal{S}}| \le \eta^{c''''(\alpha)s} |{\mathcal{S}}|,$$

\noindent completing the proof.

\appendix

\section{Proof of Lemma \ref{lemma:lemmas:random:3}}

 Without loss of generality, we
just work with the real part of $\widehat{1}_R$. We shall prove $\P_R(\sup_{\xi\neq 0}|\Re
\widehat{1}_R(\xi)|\ge \lambda/N)=o(1)$ for some appropriately
chosen $\lambda$. Since the treatment for other cases is similar, we
just show that $\P(\Re \widehat{1}_R(\xi)\ge \lambda/N)$ is
very small for each fixed $\xi\neq 0$. For convenience, put

$$X=N\Re \widehat{1}_R(\xi) = \sum_{v\in V}1_R(v)\Re e(-\langle v,\xi \rangle):=\sum_{v\in V}X_v.$$

\noindent One observes that $X$ is a sum of $N$ linearly independent real variables
$X_v$'s. Choosing $t$ to be a positive number smaller than 1, we have

\begin{align*}
\P_R(X \ge \lambda)&=\P_R(\exp(tX) \ge \exp(t\lambda))\\
&\le \E(\exp(tX))/\exp(t\lambda)\\
&=\prod\E(\exp(tX_v))/\exp(t\lambda)\\
&=\exp(t\E X)\prod \E(\exp(tX_v-t\E(X_v)))/\exp(t\lambda)\\
&=\exp(t\E X)\prod \E(\exp(tY_v))/\exp(t\lambda),
\end{align*}

\noindent where $Y_v:=X_v-E(X_v)=(1_R(v)-q) \Re e(-\langle v,\xi \rangle).$

 Notice that $|Y_v|\le 1$ and $0< t\le 1$. We thus have $\exp(tY_v)\le 1+tY_v+t^2Y_v^2$. Hence

$$\E(\exp(tY_v))\le 1+\E(t^2Y_v^2) \le \exp(\E(t^2Y_v^2)).$$

\noindent Also, because $\E X = q \Re \sum_{v\in V}e(-\langle v,\xi \rangle) =0$, it follows that

$$\P(X \ge \lambda)\le \prod \E(\exp(tY_v))/\exp(t\lambda)\le \exp(t^2 \sum_{v\in V}\E(Y_v^2))/\exp(t\lambda).$$

\noindent On the other hand, it is clear from the definition of $Y_v$ that $\sum_{v\in
V} \E(Y_v^2)\le qN.$ Thus

$$\P(X>\lambda)\le \exp(t^2qN-t\lambda).$$

\noindent By choosing $\lambda = |R|/\log N$ and
$t=\lambda/(2qN)=1/(2\log N)$ (thus $t<1$), we deduce that

$$P(X\ge |R|/\log N)\le \exp(-|R|/(2\log^2 N).$$

\noindent Hence

$$\P(\sup_{\xi \neq 0}\Re \widehat{1}_S(\xi)>|R|/(N\log N)) \le
N\exp(-|R|/(2\log^2 N))=o(1).$$

\noindent (Note that the choice for $\lambda$ above is not optimal,
but it is enough for our goal.)

\vskip .2in

\noindent {\bf Acknowledgement}. The author would like to thank prof. Van H. Vu for discussions and encouragement. He is also grateful to Philip M. Wood for reading the note and giving many valuable suggestions.

\end{document}